\documentclass{elsarticle}
\makeatletter
\def\ps@pprintTitle{%
 \let\@oddhead\@empty
 \let\@evenhead\@empty
 \def\@oddfoot{\centerline{\thepage}}%
 \let\@evenfoot\@oddfoot}
\makeatother

\usepackage{amsmath, amssymb, amsthm}
\usepackage{array}
\usepackage{graphics}
\usepackage{epsfig}
\usepackage{color}
\usepackage{booktabs}
\usepackage{caption}
\usepackage{subcaption}

\def\ta{\tilde{a}}
\def\tb{\tilde{b}}
\def\tbfx{\tilde{\bfx}}

\def\Sage{\texttt{Sage}}

\def\ud{{\rm d}}

\def\deg{{\rm deg}}

\def\rank{{\rm rank}}

\def\Bbb#1{\mathbb{#1}}

\def\CCC{\mathcal{C}}
\def\LLL{\mathcal{L}}
\def\PPP{\mathcal{P}}

\def\bfa{\mathbf{a}}
\def\bfb{\mathbf{b}}

\def\bfi{\mathbf{i}}

\def\bfx{\mathbf{x}}

\def\CC{\mathbb{C}}

\def\PP{\mathbb{P}}
\def\RR{\mathbb{R}}

\newtheorem{theorem}{{\bf Theorem}}

\newtheorem{corollary}[theorem]{{\bf Corollary}}
\newtheorem{proposition}[theorem]{{\bf Proposition}}
\newtheorem{lemma}[theorem]{{\bf Lemma}}

\pdfoutput=1

\begin{document}
\begin{frontmatter}

%\copyrightspace

\title{Detecting Symmetries\\ of Rational Plane and Space Curves}

%%%%%%%%%%%%%%%%%%%%%%%%%%%%%%%%%%%%%%%%%%%%%%%%%%%%%%%%%%%%%%%%%%%%%%%%%%%%%%%%%%%%%%%%%%%%%%%%%
%% NOTA PARA MI: CUANDO VAYAMOS A ESCRIBIR PAPER, QUITAR COMENTARIOS DESDE AQUI...................

\author[a]{Juan Gerardo Alc\'azar\fnref{proy}}
\ead{juange.alcazar@uah.es}
\author[a]{Carlos Hermoso}
\ead{carlos.hermoso@uah.es}
\author[b]{Georg Muntingh\fnref{grant}}
\ead{georgmu@math.uio.no}

\address[a]{Departamento de F\'{\i}sica y Matem\'aticas, Universidad de Alcal\'a, E-28871 Madrid, Spain}
\address[b]{SINTEF ICT, PO Box 124 Blindern, 0314 Oslo, Norge, and\\ Department of Mathematics, University of Oslo, PO Box 1053, Blindern, 0316 Oslo, Norge}

\fntext[proy]{Supported by the Spanish ``Ministerio de Ciencia e Innovacion" under the Project MTM2011-25816-C02-01. Member of the Research Group {\sc asynacs} (Ref. {\sc ccee2011/r34})}

\fntext[grant]{Partially supported by the Giner de los R\'ios grant from the Universidad de Alcal\'a.}

%%%% HASTA AQUI...............................................................
%%%%%%%%%%%%%%%%%%%%%%%%%%%%%%%%%%%%%%%%%%%%%%%%%%%%%%%%%%%%%%%%%%%%%%%%%%%%%%

\begin{abstract}
This paper addresses the problem of determining the symmetries of a plane or space curve defined by a rational parametrization. We provide effective methods to compute the involution and rotation symmetries for the planar case. As for space curves, our method finds the involutions in all cases, and all the rotation symmetries in the particular case of Pythagorean-hodograph curves. Our algorithms solve these problems without converting to implicit form. Instead, we make use of a relationship between two proper parametrizations of the same curve, which leads to algorithms that involve only univariate polynomials. These algorithms have been implemented and tested in the \Sage~system.
\end{abstract}

\end{frontmatter}

\section{Introduction}\label{sec:introduction}
\noindent The problem of detecting the symmetries of a curve has been studied extensively, mainly because of its applications in Pattern Recognition, Computer Graphics and Computer Vision.

In Pattern Recognition, a common problem is how to choose, from a database of curves, the one which best suits a given object, represented by means of an equation \cite{Huang, Lei, Sener, TC00, Tasdizen, Taubin1}. Before a comparison can be carried out, one must bring the shape that needs to be identified into a canonical position. Thus it becomes necessary to compute the symmetries of the studied curve. In this context, the computation of symmetries has been addressed using splines \cite{Huang}, by means of differential invariants \cite{Boutin, Calabi, Weiss}, using a complex representation of the implicit equation of the curve \cite{LR08, LRTh, TC00}, and using moments \cite{Huang, Suk1, Suk2, Taubin2}.

In Computer Graphics, the detection of symmetries and similarities is important, both in the 2D and the 3D case, to gain understanding when analyzing pictures, and also in order to perform tasks like compression, shape editing or shape completion. Many techniques involve statistical methods and, in particular, clustering; see for example the papers \cite{Berner08, Bokeloh, Mitra06, Podolak}, where the technique of transformation voting is used. Other techniques are robust auto-alignment \cite{Simari}, spherical harmonic analysis \cite{Martinet}, primitive fitting \cite{Schnabel}, and spectral analysis \cite{Lipman}, to quote a few.

In Computer Vision, symmetry is important for object detection and recognition. In this context, an analysis has been carried out using the Extended Gauss Image \cite{Sun} and using feature points \cite{Loy}. In addition, there are algorithms for computing the symmetries of 2D and 3D discrete objects \cite{Alt88, Brass, Jiang, Li08} and for boundary-representation models \cite{Li08, Li10, Tate}.

In the case of discrete objects (polygons, polyhedra), the symmetries can be determined exactly \cite{Alt88, Brass, Jiang, Li08}. This can be generalized to the case of more complicated shapes whose geometry is described by a discrete object, as done in \cite{Brass}, where an efficient algorithm is provided. Examples of this situation appear with B\'ezier curves and tensor product surfaces, where the shape follows from the geometry of the control points. However, in almost all of the other above references, the goal is to find \emph{approximate} symmetries of the shape. This is perfectly adequate in many applications, because the input is often a `fuzzy' shape, with missing or occluded parts in some cases. In fact, even if the input is exact, it is often an approximate, simplified model of a real object. Here we shall consider a different perspective. We assume that our input is exact, and we want to deterministically detect the existence and nature of its symmetries, without converting to implicit form. More precisely, our input will be either a plane or a space curve $\CCC$ defined by means of a rational parametrization with integer coefficients. Our goal is to (1) determine whether $\CCC$ has any symmetries, and (2) determine all symmetries in the affirmative case.

Notice that since we are dealing with a {\it global} object, i.e., the whole curve $\CCC$, we do not have a control polygon from which the geometry of the curve, and in particular its symmetries, can be derived. This could be the case if we were addressing a piece of $\CCC$, at least when $\CCC$ admits a polynomial parametrization. In that situation, $\CCC$ could be brought into B\'ezier form, and then an algorithm like \cite{Brass} could be applied. In fact, in that case the algorithm of \cite{Brass} would be computationally more effective than ours, since essentially the analysis follows from a discrete object. However, this idea is no longer applicable when the whole curve is considered.

Additionally, an analysis of \emph{approximate} symmetries of rational curves could be attempted by sampling points on the curve, and then applying algorithms like  \cite{Brass, Li08, Li10}. In that case, the question is how to choose suitable zones for sampling, which amounts to collecting some information on the shape of the curve \cite{ADT}. A natural strategy is to look for notable points on the curve, like singularities, inflection points or vertices: Since any symmetry maps notable points of a certain nature to the same kind of points or leaves them invariant, one might sample around these points.
One thus obtains clusters of points that must be compared. There would be various possibilities for comparing these clusters, depending on the kind of symmetry one is looking for, all which should be explored. Still, this approach only leads to an approximate estimate on the existence of symmetries, which is a different problem than the one considered in this paper.

Up to our knowledge, the deterministic problem for whole curves has only been solved in the case of implicit plane curves \cite{LR08, LRTh} and in the case of polynomially parametrized plane curves \cite{A13}. The case of space curves seems absent from the literature. In \cite{LR08, LRTh}, the authors provide an elegant method to detect rotation symmetry of an implicitly defined algebraic curve, and efficiently find the exact rotation angle and rotation center. The method uses a complex representation $F(z,\bar{z})=0$ of the curve. Some cases not treated in \cite{LR08} are completed in \cite{LRTh}, where similar ideas are applied to detect mirror symmetry. In contrast, our method applies directly to the parametrization, which is the most common representation in CAGD, avoiding the conversion into implicit form. The approach in \cite{A13} is similar to ours, although it should be noted that restricting to polynomial parametrizations yields an advantage for solving the problem fast and efficiently.

The main ingredient in our method is the underlying relation between two parametrizations of a curve that are \emph{proper}, i.e., injective except perhaps for finitely many values of the parameter. Essentially, whenever a symmetry is present, this symmetry induces an alternative parametrization of the curve. Furthermore, if the starting parametrization is proper, this second parametrization is also proper. Since two proper parametrizations of a same curve are related by means of a M\"obius transformation \cite{SWPD}, we can reduce the problem to finding this transformation. Thus, \emph{involutions}, i.e., symmetries with respect to a point, line or plane, can be detected and determined for plane and space curves. For rotations, we need one more ingredient: a formulation in terms of complex numbers for plane curves, or the Pythagorean-hodograph assumption for space curves. In practice, our methods boil down to computing greatest common divisors and finding real roots of univariate polynomials, which are tasks that can be performed efficiently.

\section{Symmetries of plane and space curves}\label{sec:generalities}
\noindent Throughout the paper we shall consider a rational curve $\CCC\subset \RR^n$, where $n=2$ or $n=3$, neither a line nor a circle, defined by means of a proper rational parametrization
\begin{equation}\label{eq:C}
\bfx:\RR \dashrightarrow \CCC\subset \RR^n,\qquad \bfx(t)= \big(x_1(t),x_2(t),\ldots,x_n(t)\big),
\end{equation}
where
\[ x_i(t) = \frac{p_i(t)}{q_i(t)}, \qquad p_i,q_i\in \RR[t],\qquad \gcd(p_i, q_i) = 1,\qquad i=1,\ldots,n. \] Here ``$\gcd$'' refers to the greatest common divisor. Since $\CCC$ is rational, it is irreducible. One can check whether a parametrization of a plane curve is proper, and every rational plane curve can be properly reparametrized without extending the ground field. See \cite{SWPD} for a thorough study on properness and a proof of these claims, and see \cite[\S 3.1]{A12} for similar results for rational space curves.

We recall some facts from Euclidean geometry \cite{Coxeter}. An \emph{isometry} of $\RR^n$ is a map $f:\RR^n\longrightarrow \RR^n$ preserving Euclidean distances. Any isometry $f$ of $\RR^n$ is linear affine, taking the form \begin{equation}\label{eq:isometry}
f(\bfx) = Q\bfx + \bfb, \qquad \bfx\in \RR^n,
\end{equation}
with $\bfb \in \RR^n$ and $Q\in \RR^{n\times n}$ an orthogonal matrix. In particular $\det(Q) = \pm 1$. The isometries of the plane and space form a group under composition that is generated by reflections, i.e., symmetries with respect to a hyperplane, or \emph{mirror symmetries}. An isometry is called \emph{direct} when it preserves the orientation, and \emph{opposite} when it does not. In the former case $\det(Q) = 1$, while in the latter case $\det(Q) = -1$. The identity map $\textup{id}_{\RR^n}$ of $\RR^n$ is called the \emph{trivial symmetry}. An isometry $f(\bfx) = Q\bfx + \bfb$ of $\RR^n$ is called an \emph{involution} if $f\circ f = \textup{id}_{\RR^n}$, in which case $Q^2 = I$ is the identity matrix and $\bfb\in\ker(Q+I)$.

The nontrivial isometries of the Euclidean plane are classified into reflections, rotations, translations, and glide reflections. The special case of \emph{central symmetries} is of particular interest and corresponds to a rotation by an angle~$\pi$. Central and mirror symmetries are involutions.

The classification of the nontrivial isometries of Euclidean space again includes reflections (in a plane), rotations (about an axis), and translations, and these combine in commutative pairs to form twists, glide reflections, and rotatory reflections. Composing three reflections in mutually perpendicular planes through a point $P$, yields a \emph{central inversion} with center $P$, i.e., a symmetry with respect to the point $P$. The special case of rotation by an angle $\pi$ is again of special interest, and it is called an \emph{axial symmetry}. Central inversions, reflections, and axial symmetries are involutions.

By B\'ezout's theorem, an algebraic curve other than a line cannot be invariant under a translation or glide reflection, and a space curve can, in addition, not be invariant under a twist. We shall refer to the remaining isometries as \emph{symmetries}, and we shall say that a plane or space curve $\CCC$ is \emph{symmetric}, if it is invariant under a nontrivial symmetry. Any algebraic curve in the plane, neither a line nor a circle, has finitely many symmetries \cite[\S 5]{LRTh}. We need the following lemma to show the same result for \emph{nondegenerate} space curves, i.e., space curves not contained in a plane.

\begin{lemma} \label{lem:angles}
Let $\CCC\subset \RR^3$ be a nondegenerate irreducible space curve, invariant under a rotation with axis $\LLL$ and angle $\theta$. Then $\theta = 2\pi/k$, with $k\leq \deg(\CCC)$ an integer.
\end{lemma}

\begin{proof}
For any plane $\Pi$ normal to $\LLL$, a rotation about $\LLL$ induces a rotation of the same angle on $\Pi$ around the point $P := \LLL\cap \Pi$. But then $\theta = 2\pi/k$, with $k$ an integer that is at most the number of points in the intersection $\CCC\cap \Pi$; however, this is at most $\deg(\CCC)$ by definition of the degree of a nondegenerate irreducible curve.
\end{proof}

\begin{lemma}\label{lem:ParallelRotations}
Let $\CCC\subset \RR^3$ be an irreducible space curve, invariant under two rotations $f_1, f_2$ with axes $\LLL_1, \LLL_2$. Then $\LLL_1, \LLL_2$ cannot be parallel.
\end{lemma}

\begin{proof}
Suppose that $\LLL_1, \LLL_2$ are parallel. Let $\Pi$ be a plane normal to $\LLL_1, \LLL_2$ that intersects $\CCC$ in at least one point. The set $\CCC\cap \Pi$ is invariant under both rotations. But if a set of planar points exhibits rotation symmetry, then the rotation center must be the barycenter of the points, implying that $\LLL_1 = \LLL_2$.
\end{proof}

\begin{proposition} \label{finite-involutions}
Let $\CCC\subset \RR^3$ be a space curve different from a line or a circle. Then $\CCC$ is invariant under at most:
\begin{itemize}
\item [(i)] one central inversion;
\item [(ii)] finitely many rotation symmetries, whose axes are all concurrent;
\item [(iii)] finitely many mirror symmetries, whose planes share a point.
\end{itemize}
\end{proposition}

\begin{proof}
This result is known to hold when $\CCC$ is degenerate \cite{LRTh}, so assume that $\CCC$ is nondegenerate.

(i): If $\CCC$ is invariant under two central inversions with symmetry centers $P_1$ and $P_2$, then it is invariant under their composition, which is a translation by $2\|P_1 - P_2\|$ along the direction $P_1 P_2$ \cite[\S 7.3]{Coxeter}. Since $\CCC$ is not a line, it cannot be invariant under a nontrivial translation, implying that $P_1 = P_2$.

(ii): The composition of two rotations with axes $\LLL_1, \LLL_2$ is: (a) when the axes are parallel, a rotation with axis parallel to $\LLL_1,\LLL_2$; (b) when the axes intersect, a rotation with axis passing through $\LLL_1\cap \LLL_2$; (c) when the axes are skew, a twist. We can discard the cases (a) (by Lemma \ref{lem:ParallelRotations}) and (c).

In the remaining case (b), if $\CCC$ is invariant under three rotations with axes intersecting pairwise in three distinct points forming a plane $\Pi$, then the composition of any two rotations with axes $\LLL_1, \LLL_2$ yields a rotation with axis $\LLL$ intersecting $\Pi$ transversally in a point away from the third axis $\LLL_3$. But then the axes $\LLL$ and $\LLL_3$ are skew, which is case (c) and cannot happen.

Finally, suppose we have an infinite number of rotation axes $\{\LLL_i\}$ meeting in a point $P$. The set of lines through $P$ forms a real projective plane $\PP^2(\RR)$, which is compact. The points $\{\PPP_i\}\subset \PP^2(\RR)$ corresponding to the axes $\{\LLL_i\}$ will therefore have a point of accumulation $\PPP$. Any neighbourhood of $\PPP$ will contain an infinite number of points in $\{\PPP_i\}$, corresponding to an infinite number of axes in $\{\LLL_i\}$. These axes meet in infinitely many distinct angles. The composition of two rotations with concurrent axes is another rotation, about an axis perpendicular to the concurrent rotation axes. If the rotations have rotation angles $\alpha, \beta$ and their axes meet with an angle $\Phi$, then the composition is a rotation by an angle $\gamma$, where \cite{Heard}
\[\mbox{cos}\left(\frac{\gamma}{2}\right)=\mbox{cos}\left(\frac{\alpha}{2}\right)\cdot \mbox{cos}\left(\frac{\beta}{2}\right)-\mbox{sin}\left(\frac{\alpha}{2}\right)\cdot \mbox{sin}\left(\frac{\beta}{2}\right)\cdot \mbox{cos}(\Phi).\]
Since there are finitely many $\alpha,\beta$ by Lemma \ref{lem:angles} but infinitely many angles $\Phi$, we get infinitely many angles $\gamma$ as well, therefore contradicting Lemma \ref{lem:angles}. We conclude that $\CCC$ has at most finitely many rotation symmetries, whose axes are concurrent.

(iii): The composition of two mirror symmetries with planes $\Pi_1, \Pi_2$ is: (a) if the planes are parallel, a translation; (b) if the planes intersect, a rotation about $\Pi_1\cap \Pi_2$ making twice the angle as between $\Pi_1, \Pi_2$. Case (a) can be discarded, so all mirror symmetries of $\CCC$ have intersecting planes, and the statement follows in the case that $\CCC$ has at most two mirror symmetries.

Suppose $\CCC$ has at least three mirror symmetries $f_1, f_2, f_3$ with corresponding planes $\Pi_1, \Pi_2, \Pi_3$. If these three planes intersect in a point at infinity, then any two pairs, say $(\Pi_1, \Pi_2)$ and $(\Pi_1, \Pi_3)$, intersect in parallel lines $\LLL := \Pi_1\cap \Pi_2$ and $\LLL' := \Pi_1\cap \Pi_3$ in the finite plane, which cannot happen by Lemma \ref{lem:ParallelRotations}. It follows that the three planes intersect in a point $P:=\Pi_1\cap\Pi_2\cap\Pi_3$ in the finite plane.

Suppose there is a fourth mirror symmetry $f_4$ with plane $\Pi_4$ not containing $P$. Then $\Pi_4$ intersects $\Pi_3$ in a certain line $\LLL''$. If $\LLL''$ does not pass through $P$, then $\CCC$ has rotation symmetries about the two skew axes $\LLL$ and $\LLL''$. But then $\CCC$ would be invariant under their composition, which is a twist, and this cannot happen. So $\Pi_4$ also contains $P$ and (iii) holds.
\end{proof}

\begin{corollary} \label{invol-finite}
The number of symmetries of a plane or space curve, other than a line or a circle, is finite.
\end{corollary}

The following theorem forms the foundation of our method. We need the definition of a \emph{M\"obius transformation} (on the affine real line), which is a rational function
\begin{equation}\label{eq:Moebius}
\varphi: \RR \dashrightarrow \RR,\qquad \varphi(t) = \frac{a t + b}{c t + d},\qquad \Delta := ad - bc\neq 0.
\end{equation}
In particular the identity map is a M\"obius transformation, which we refer to as the \emph{trivial} transformation.

\begin{theorem} \label{fundam-result}
The curve $\CCC$ in \eqref{eq:C} is invariant under a nontrivial symmetry $f$ of the form \eqref{eq:isometry} if and only if there exists a nontrivial M\"obius transformation $\varphi$, with real coefficients $a,b,c,d$, such that
\begin{equation}\label{eq:fundam}
f\big(\bfx(t)\big) = \bfx\big(\varphi(t)\big).
\end{equation}
Moreover, for any $f$ there is a unique M\"obius transformation $\varphi$ satisfying \eqref{eq:fundam}.
\end{theorem}
\begin{proof} If there are two M\"obius transformations $\varphi_1, \varphi_2$ satisfying \eqref{eq:fundam}, then $\bfx(\varphi_1) = \bfx(\varphi_2)$. Since $\bfx$ is proper, it follows that $\varphi_1 = \varphi_2$. For the first claim:

``$\Longrightarrow$'': Let $\tilde\bfx(t) := f\big(\bfx(t)\big)$. Since $\bfx$ is proper, $\bfx^{-1}$ is defined for all but finitely many values, and $\bfx^{-1} \circ \tilde\bfx$ is a birational map $\varphi := \tilde\bfx^{-1} \circ \bfx$ from the real line to itself. Any such map lifts to a birational automorphism of the complex projective line $\PP(\CC)$, which are known to be M\"obius transformations \cite{SWPD}. It follows that $\varphi$ takes the form \eqref{eq:Moebius} and is nontrivial because $f$ is nontrivial. Moreover, since $\varphi$ maps the real line to itself, we can assume that the coefficients of $\varphi(t)$ are real.

``$\Longleftarrow$'': If $f\big(\bfx(t)\big) =\bfx\big(\varphi(t)\big)$ for some nontrivial isometry $f$ and nontrivial M\"obius transformation $\varphi$, we observe that $\tilde\bfx(t) := f\big(\bfx(t)\big)$ is an alternative parametrization of $\CCC$, and therefore that $\CCC$ is invariant under $f$.
\end{proof}

Equation \eqref{eq:fundam} relates the symmetries of $\CCC$ to M\"obius transformations in the parameter domain. We obtain the following lemma.

\begin{lemma}\label{lem:order}
Suppose an isometry $f$ and M\"obius transformation $\varphi$ are related by \eqref{eq:fundam}. For any integer $k$, the composition $f^k = \textup{id}_{\RR^n}$ if and only if $\varphi^k = \textup{id}_{\RR}$.
\end{lemma}
\begin{proof}
Because $\bfx$ is proper, its inverse $\bfx^{-1}$ exists as a rational map, and $\varphi=\bfx^{-1} \circ f \circ \bfx$. The result follows from $\varphi^k = (\bfx^{-1} \circ f \circ \bfx)^k = \bfx^{-1} \circ f^k \circ \bfx$.
\end{proof}

At this point one could in principle find the symmetries of $\CCC$ by determining $f$ and $\varphi$ satisfying the equation \eqref{eq:fundam}. However, the resulting polynomial system would involve too many variables in the coefficients of $\varphi$ and $f$ to be solved efficiently. In the following section we propose an efficient method to determine the symmetries of $\CCC$.

\section{Determining symmetries of plane and space curves} \label{sec:determinesym}
\noindent To consider plane and space curves in one go, we embed $\RR^2$ into $\RR^3$ as the plane of points with zero third component. The mappings on $\RR^2$ are lifted to mappings of $\RR^3$ leaving the third component invariant.

For technical reasons, we assume that $\bfx(0)$ is well defined, and that $\bfx'(0),\bfx''(0)$ are also well defined, nonzero, and not parallel. Notice that this amounts to requiring that the \emph{curvature}
\[
\kappa(t) := \frac{\|\bfx'(t) \times \bfx''(t)\|}{\|\bfx'(t)\|^3}
\]
is well defined and nonzero at $t = 0$.
Since this is the case for almost all parameters $t$, this condition holds after applying an appropriate, even random, linear affine change of the parameter $t$.

Assume that the plane or space curve $\CCC$ in \eqref{eq:C} is invariant under a nontrivial symmetry $f(\bfx) = Q\bfx + \bfb$. By Theorem~\ref{fundam-result}, there is a M\"obius transformation~$\varphi$ satisfying \eqref{eq:fundam}. Our strategy will be to first, in Sections \ref{sec:dzero}--\ref{sec:spacerotations}, express all unknown parameters in $Q, \bfb$, and $\varphi$ as rational functions of a single parameter $b$ of $\varphi$. Substituting these rational functions into \eqref{eq:fundam} and clearing denominators, one obtains three polynomials in $t$, whose coefficients are polynomials in $b$. For \eqref{eq:fundam} to hold identically for all $t$, each of these polynomial coefficients must be zero, which happens if and only if their greatest common divisor vanishes. Removing from this polynomial all factors for which the M\"obius transformation or symmetry is not defined or not invertible, one obtains a polynomial $P(b)$ in which every real root corresponds to a symmetry.

Let $f$ be a plane rotation or an involution in $\RR^2$ or $\RR^3$. Alternatively, let $f$ be any isometry and $\CCC$ be a Pythagorean-Hodograph curve. We can now formulate the main theorem of the paper, which will be proved case-by-case in Sections \ref{sec:dzero}--\ref{sec:spacerotations}.

\begin{theorem} \label{th-alg}
The curve $\CCC$ has a nontrivial symmetry $f$ if and only if $P(b)$ has a real root $b$ at which the parameters of $Q, \bfb$, and $\varphi$ are well defined.
\end{theorem}

Each real root of $P(b)$ determines a M\"obius transformation, which corresponds uniquely to a symmetry of $\CCC$ by Theorem \ref{fundam-result}. By Corollary \ref{invol-finite}, $\CCC$ has at most finitely many symmetries, implying that $P(b)$ cannot be identically zero.

Finally, observe that one can directly find the symmetry type and its elements by analyzing the set of fixed points of the symmetry $f(\bfx) = Q(b) \bfx + \bfb(b)$. In particular, $\rank\big(Q(b) - I\big)$ is 1 for a symmetry with respect to a plane; 2 for a symmetry with respect to a line; 3 for a rotation symmetry or central inversion.

\subsection{The case $d = 0$}\label{sec:dzero}
\noindent If $d = 0$, equation \eqref{eq:fundam} becomes
\[ Q\bfx(t) + \bfb = \bfx\big(\varphi(t)\big) = \bfx\big(\ta / t + \tb \big), \]
where $\ta := b/c$ and $\tb := a/c$. Applying the change of variables $t\longrightarrow 1/t$ and writing $\tbfx(t) := \bfx(1/t)$, we obtain
\begin{equation} \label{spec}
 Q\tbfx(t) + \bfb =  \bfx\big(\ta t + \tb\big).
\end{equation}
Without loss of generality, we assume that $\tbfx(t)$ is well defined at $t=0$ and that $\tbfx'(0),\tbfx''(0)$ are well defined, nonzero, and not parallel. Evaluating \eqref{spec} at $t = 0$ yields
\begin{equation}\label{eq:evaldzero}
Q\tbfx(0) + \bfb = \bfx(\tb),
\end{equation}
while differentiating once and twice and evaluating at $t=0$ yields
\begin{equation}\label{eq:diff12dzero}
Q\tilde{\bfx}'(0) =  \bfx'(\tb)\cdot \ta,\qquad Q\tilde{\bfx}''(0) =  \bfx''(\tb)\cdot \ta^2.
\end{equation}

Taking inner products and using that $Q$ is orthogonal, we get
\begin{equation}\label{eq:a234}
\ta^2 = \frac{ \|\tilde{\bfx}'(0)\|^2 }{\|\bfx'(\tb)\|^2},\qquad
\ta^3 = \frac{ \langle \tilde{\bfx}'(0), \tilde{\bfx}''(0)\rangle }{\langle \bfx'(\tb), \bfx''(\tb)\rangle},
\end{equation}
from which we can write $\ta$ as a rational function of $\tb$,
\begin{equation}\label{spec5}
\ta = \frac{\Vert \bfx'(\tb)\Vert^2}{\Vert \tilde{\bfx}'(0)\Vert^2}\cdot
\frac{\langle\tilde{\bfx}'(0),\tilde{\bfx}''(0)\rangle}{\langle\bfx'(\tb),\bfx''(\tb)\rangle}.
\end{equation}
A straightforward, but lengthy, calculation yields
\begin{equation}\label{eq:cross}
(M\bfa) \times (M\bfb) = \det(M) M^{-T} (\bfa \times \bfb)
\end{equation}
for any invertible matrix $M\in \RR^{3\times 3}$ and vectors $\bfa, \bfb\in \RR^3$. Taking the cross product in \eqref{eq:diff12dzero} and using that $Q$ is orthogonal, one obtains
\begin{equation}\label{eq:diff1xdiff2dzero}
Q\big( \tbfx'(0)\times \tbfx''(0) \big) = \det(Q) \big( \bfx'(\tb) \times \bfx''(\tb) \big) \cdot \ta^3.
\end{equation}

We analyze separately the cases $n = 2$ and $n = 3$. For $n = 2$, \eqref{eq:diff12dzero} implies that multiplying $Q$ by the matrix $A := [\tbfx'(0), \tbfx''(0)]$ gives the matrix $B := [\bfx'(\tb) \ta, \bfx''(\tb) \ta^2]$ so that $ Q = BA^{-1}$. For $n = 3$, multiplying $Q$ by the matrix $A := [\tbfx'(0), \tbfx''(0), \tbfx'(0) \times \tbfx''(0)]$ gives the matrix
\[ B := \big[\bfx'(\tb) \ta, \bfx''(\tb) \ta^2, \det(Q) \big( \bfx'(\tb) \times \bfx''(\tb) \big) \cdot \ta^3\big] \]
and $Q = BA^{-1}$. One sets $\det(Q) = 1$ to find the direct transformations and $\det(Q) = -1$ to find the opposite transformations. Substituting $\ta = \ta(\tb)$, one expresses $Q$ as a matrix-valued rational function of $\tb$. Finally from \eqref{eq:evaldzero} one expresses $\bfb$ as a vector-valued rational function of $\tb$.

\subsection{The case $d \neq 0$}\label{sec:dnonzero}
\noindent If $d\neq 0$, we may and will assume $d = 1$ after scaling the coefficients of $\varphi$ if necessary. Differentiating \eqref{eq:fundam} twice, we get
\begin{align}
Q\bfx'(t)  & = \bfx'\big(\varphi(t)\big)\cdot\varphi'(t) = \bfx'\left(\frac{a t + b}{c t + d}\right)\frac{\Delta}{(c t + d)^2}, \label{eq:first}\\
Q\bfx''(t) & = \bfx''\big(\varphi(t)\big)\big(\varphi'(t)\big)^2 + \bfx'\big(\varphi(t)\big)\varphi''(t) \label{eq:second}\\
           & =\displaystyle{\bfx''\left(\frac{a t + b }{c t + d }\right) \frac{\Delta^2}{(c t + d )^4} - 2\bfx'\left(\frac{a t + b}{c t + d}\right) \frac{c \Delta}{(c t + d)^3}}. \notag
\end{align}
Evaluating \eqref{eq:first} and \eqref{eq:second} at $t = 0$ yields
\begin{align}
Q \bfx'(0)  & = \bfx'(b) \Delta, \label{eq:diff1}\\
Q \bfx''(0) & = \bfx''(b) \Delta^2 - 2 \bfx'(b) c\Delta. \label{eq:diff2}
\end{align}
From \eqref{eq:diff1}, and using that $Q$ is orthogonal, we can express
\begin{equation}\label{eq:deltasquared} \Delta^2 = \frac{\|\bfx'(0) \|^2}{\|\bfx'(b) \|^2}\end{equation}
solely in terms of $b$. Taking the cross product of \eqref{eq:diff1}, \eqref{eq:diff2} and using again \eqref{eq:cross} and that $Q$ is orthogonal, one obtains
\begin{equation}\label{eq:diff1xdiff2}
Q \big(\bfx'(0) \times \bfx''(0) \big)  = \det(Q)\cdot \Delta^3 \big(\bfx'(b) \times \bfx''(b) \big).
\end{equation}
Using \eqref{eq:deltasquared} and taking norms, one reaches
\begin{equation}\label{cond}
 \|\bfx'(0) \times \bfx''(0)\|^2 \|\bfx'(b) \|^6 - \|\bfx'(0) \|^6 \|\bfx'(b) \times \bfx''(b)\|^2 = 0,
\end{equation}
which amounts to $\kappa^2(b) = \kappa^2(0)$. If this equation does not have a real root, then we know that $\CCC$ does not have any symmetry of any type. If it does, then we proceed to write $Q$ in terms of $b$. By computing the dot product of \eqref{eq:diff1}, \eqref{eq:diff2} and using that $Q$ is orthogonal, we get
\begin{equation}\label{quat-4}
c = - \frac{\langle \bfx''(0), \bfx'(0) \rangle}{2\|\bfx'(0)\|^2} + \Delta \frac{\langle \bfx''(b), \bfx'(b)\rangle}{2\|\bfx'(b)\|^2}.
\end{equation}

Next we consider separately involutions and plane rotations for which the corresponding M\"obius transformation has parameter $d = 1$.

\subsection{Involutions} \label{sec:involutions}
\noindent Assume that $f(\bfx) = Q\bfx + \bfb$ is a nontrivial involution. Then
\[ \bfx = f^2(\bfx) = Q(Q\bfx + \bfb) + \bfb = Q^2\bfx + (Q+I)\bfb,\qquad \bfx\in \RR^n, \]
implying that $Q^2=I$. By Lemma \ref{lem:order}, $\varphi^2 = \textup{id}_{\RR}$, implying that $(a+d)b = 0$, $(a+d)c = 0$, and $a^2 = d^2$. If $a + d\neq 0$, then $b = c = 0$ and $a = d$, and therefore $\varphi=\mbox{id}_{{\Bbb R}}$, which contradicts that $f(\bfx)$ is nontrivial. Therefore $a = -d$. So, $\Delta = -1 -bc$, and from \eqref{quat-4} we can write
\begin{equation}\label{eq:gammaofb}
c = c(b) = - \frac{\langle \bfx'(b), \bfx''(b) \rangle \|\bfx'(0) \|^2 + \langle \bfx'(0), \bfx''(0) \rangle \|\bfx'(b) \|^2}{\|\bfx'(0) \|^2 \Big( b\langle \bfx'(b), \bfx''(b) \rangle + 2 \|\bfx'(b) \|^2 \Big)}.
\end{equation}
as a rational function of $b$.

By changing the parametrization if necessary, we can determine $c(b)$ by assuming that the numerator and denominator of the above fraction have no real root in common. Alternatively, we can take the gcd of the numerator and denominator and find the common real roots $b$, determine the corresponding $\Delta(b)$ from \eqref{eq:deltasquared} by considering both signs separately, and $c(b)$ from $\Delta(b) = -1 - b\cdot c(b)$. Moreover, the denominator of this expression vanishes iff
\[
0 = b\langle \bfx'(b), \bfx''(b) \rangle + 2 \|\bfx'(b) \|^2  = \frac{b}{2} \frac{\ud}{\ud b} \|\bfx'(b)\|^2 + 2 \|\bfx'(b) \|^2,
\]
which happens precisely when $\Vert \bfx'(b)\Vert^2 = M/b^4$, with $M$ a nonzero constant. However, in that case $\bfx'(0)$ is not defined, which contradicts one of our initial assumptions. Hence, the above expression for $c(b)$ is well defined.

Once the rational functions $\Delta = \Delta(b)$ and $c = c(b)$ are obtained, the matrix $Q = Q(b)$ can again be determined from its action on $\bfx'(0), \bfx''(0)$, and $\bfx'(0) \times \bfx''(0)$, which is given by equations \eqref{eq:diff1}, \eqref{eq:diff2}, and \eqref{eq:diff1xdiff2}. One finds $\bfb(b)$ from evaluating \eqref{eq:fundam} at $t = 0$.

\subsection{Plane rotations} \label{sec:planerotations}
\noindent In order to detect rotation symmetries in the plane, we identify the Euclidean plane with the complex plane as $(x_1, x_2) \simeq x_1 + x_2 \bfi$. Thus the parametrization $\bfx = (x_1, x_2)$ in \eqref{eq:C} yields a parametrization
\[ z: \RR\dashrightarrow \CCC\subset \CC,\qquad z(t) = x_1(t) + x_2(t)\bfi, \]
where we denoted the curve by the same symbol $\CCC$. Writing $z_0$ for the rotation center and $\theta$ for the rotation angle, equation \eqref{eq:fundam} takes the form
\begin{equation} \label{rot-1}
z\big(\varphi(t)\big) = z_0 + e^{{\bf i}\theta}\cdot \big(z(t) - z_0\big).
\end{equation}
Differentiating this expression we obtain
\begin{equation} \label{rot-diff}
z'\big(\varphi(t)\big)\cdot \frac{\Delta}{(c t + d)^2}=e^{{\bf i}\theta}\cdot z'(t).
\end{equation}

Without loss of generality, we assume that $z(t)$ is well defined at $t = 0$ and $z'(0)\neq 0$, so that evaluating \eqref{rot-1} and \eqref{rot-diff} at $t = 0$ gives
\begin{equation}\label{eq:rot-1tzero}
e^{{\bf i}\theta} = \Delta \frac{z'(b)}{z'(0)},\qquad
z_0 = \frac{z(b) - e^{\theta \bfi} z(0)}{1 - e^{\theta \bfi}},
\end{equation}
expressing the symmetry in terms of the M\"obius transformation. Differentiating \eqref{rot-diff}, evaluating at $t=0$, and solving for $c$, we deduce
\begin{equation} \label{rot-5}
c = \frac{1}{2}\left[\frac{z''(b)}{z'(b)}\Delta - \frac{z''(0)}{z'(0)} \right].
\end{equation}
Since all coefficients of $\varphi$ are real, the imaginary part of the above expression for $c$ must be zero, which yields rational expressions $\Delta = \Delta(b), c = c(b)$, and therefore also $a = a(b) = \Delta(b) + b\cdot c(b)$. The symmetry itself is determined from \eqref{eq:rot-1tzero}.

\subsection{Space rotations and Pythagorean-Hodograph curves}\label{sec:spacerotations}
\noindent Some rotation symmetries of space curves are found by the previous algorithms. Since axial symmetries are involutions, these rotations will be found directly by the method of Section \ref{sec:involutions}. By the Cartan-Dieudonn\'e Theorem, any rotation $f$ in $\RR^3$ is the composition of three reflections $f_1,f_2,f_3$. However, if our curve $\CCC$ has a rotation symmetry $f$, then these reflections $f_1,f_2,f_3$ need not be mirror symmetries of our curve. Taking compositions of the reflections found in Section \ref{sec:involutions} will therefore only yield some of the rotations of $\CCC$. In addition, the rotations whose corresponding M\"obius transformation have parameter $d=0$ will be found by the method of Section \ref{sec:dzero}.

Unfortunately it seems that the approach of the previous section can not be generalized to find all rotations of space curves. The complex numbers in the plane can be replaced by quaternions \cite{Farouki, Goldman} in space, which provides a convenient way to express rotations. For instance, one can give quaternion versions of \eqref{eq:first} and \eqref{eq:second}. The difficulty, however, comes from the fact that quaternions are not commutative, which makes it hard to eliminate the parameters defining the rotation in the resulting equations. Because of this, we have not been able to prove a version of Theorem~\ref{th-alg} for space rotations. In fact, we are uncertain whether it is possible in that case to write all parameters of the M\"obius transformation as rational functions of just one of them. While we might write these parameters in terms of two of them, which would yield a bivariate polynomial system, we feel that this solution is not satisfactory computationally. We therefore pose the question here as a pending problem.

However, if the curve $\CCC$ from \eqref{eq:C} is a \emph{Pythagorean-Hodograph curve}, i.e., if there exists a rational function $\sigma(t)$ such that
\[\Vert\bfx'(t)\Vert^2=x_1'^2(t)+x_2'^2(t)+\cdots +x_n'^2(t)=\sigma^2(t),\]
then we do not run into the same obstacle. Such curves form an important topic in Computer Aided Geometric Design, and they have been studied extensively both from the point of view of theory and of applications \cite{Farouki}. In this case we can determine all rotation symmetries, since \eqref{eq:deltasquared} gives a rational function $\Delta(b) = \pm \|\bfx'(0)\|/\sigma(b)$, from which we find a rational function $c(b)$ by \eqref{quat-4}, and finally a rational function $a(b) = \Delta(b) + b\cdot c(b)$. After this we determine the symmetry as before. The sign of $\Delta(b)$ is not easily determined in advance, so it is necessary to carry out the algorithm for both cases.

\section{Implementation and experimentation} \label{sec:implementation}
\noindent For all but the simplest examples, the computations quickly become too large to be carried out by hand, and a computer algebra system is needed. We have therefore implemented and tested the algorithms in \Sage~\cite{sage}. The resulting worksheet with implementations and examples can be downloaded from the website of the third author \cite{WebsiteGeorg}.

\begin{figure}
\centering
\begin{subfigure}[b]{0.45\textwidth}
  \centering
  \includegraphics[scale=0.7]{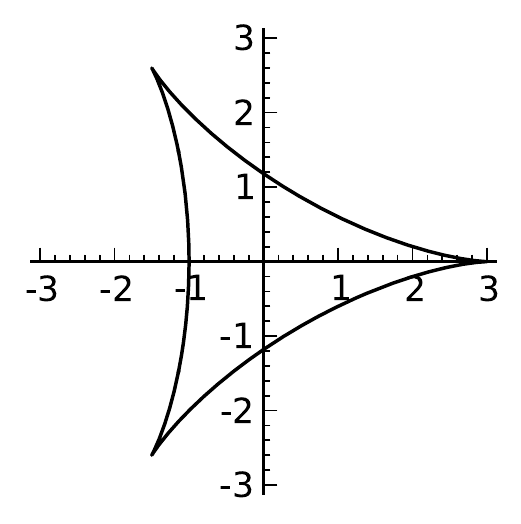}
  \caption{ }\label{fig:Deltoid}
\end{subfigure}%
\qquad
\begin{subfigure}[b]{0.45\textwidth}
  \centering
  \includegraphics[scale=0.1]{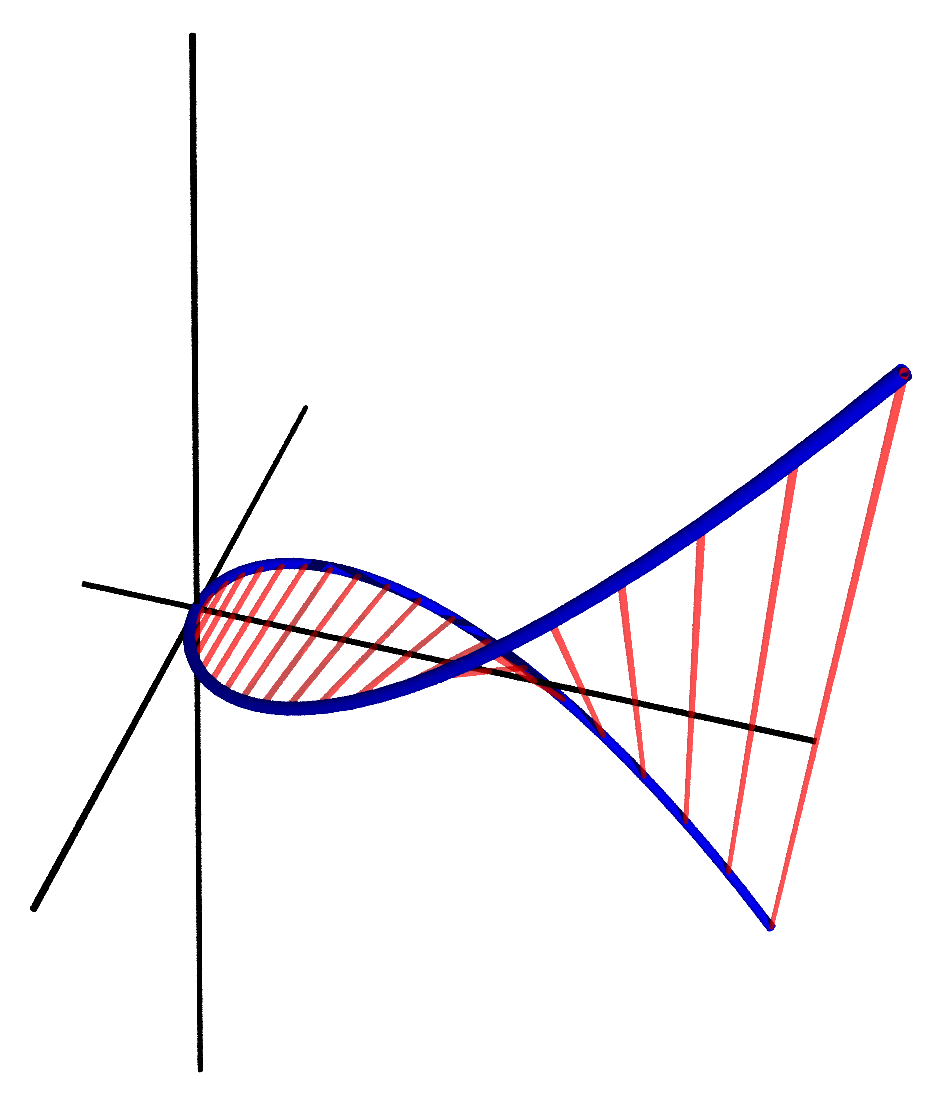}
  \caption{ }\label{fig:TwistedCubic}
\end{subfigure}%
\caption{The deltoid from \eqref{eq:deltoid} (left) and the twisted cubic from \eqref{eq:TwistedCubic} (right).}\label{fig:Curves}
\end{figure}

\subsection{An example: rotations of the deltoid}\label{sec:deltoid}
\noindent Let $\CCC$ be the \emph{deltoid} from Figure \ref{fig:Deltoid}, defined parametrically as the image of the map $z: \RR\longrightarrow \CC$,
\begin{equation}\label{eq:deltoid}
z(t) = \frac{-t^4 + 4t^3 - 12t^2 + 16t - 4}{t^4 - 4t^3 + 8t^2 - 8t + 4} + \frac{8t^3 - 24t^2 + 24t - 8}{t^4 - 4t^3 + 8t^2 - 8t + 4}\bfi.
\end{equation}
We follow the recipe from Section \ref{sec:planerotations} to find its rotations. This parametrization is well defined at $t=0$ and satisfies $z'(0) \neq 0$.
Using that the imaginary part of \eqref{rot-5} is zero, we find $\Delta(b) = \frac{1}{2}b^2 - b + 1$,
\[ c(b) = \frac{b(2b^3 - 3b^2 - 18b + 22)}{4(b^2 - 2b - 2)(1 - b)},\qquad
   a(b) = \frac{7b^4 - 34b^3 + 30b^2 + 8b - 8}{4(1 -b)(b^2 - 2b - 2)}. \]
The symmetry is determined by \eqref{eq:rot-1tzero}, as
\begin{equation}\label{eq:deltoidangle}
e^{\theta \bfi}(b) =
-\frac{(b - 1)(b^2 - 2b - 2)}{(b^2 - 2b + 2)^2}(b - 2 + b \bfi),
\end{equation}
\begin{equation}\label{eq:deltoidcenter}
z_0(b) = -3\frac{b^2 - 6b + 6}{(5b^4 - 32b^3 + 72b^2 - 64b + 20)} \big(b^2 - 4b + 2 + 2(b - 2)(b - 1)\bfi \big).
\end{equation}
Substituting these expressions into \eqref{rot-1} yields a rational function in $t$, whose coefficients are polynomials in $b$. This rational function is identically zero if and only if the gcd
$b(b - 1)(b^2 - 2b - 2)(b^2 - 6b + 6)$
of the coefficients in the numerator is zero. Removing all factors for which either the M\"obius transformation or the symmetry is not defined or not invertible, we obtain a polynomial $P(b) = b(b^2 - 6b + 6)$. Substituting its three real zeros $b = 0, 3\pm \sqrt{3}$ into \eqref{eq:deltoidangle}, \eqref{eq:deltoidcenter} we find rotations about $z_0 = 0$ with angles $\theta = 0, 2\pi/3, 4\pi/3$. There are no additional symmetries for the case $d = 0$.

\begin{table}
\begin{tabular*}{\columnwidth}{m{5.5em}@{\extracolsep{\stretch{1}}}*{5}{ccccccc}@{}}
\toprule
curve & deg. & parametrization $z$ & $u(t)$ & \#$_\text{rot}$ & \#$_\text{inv}$ & $t_\text{rot}$ & $t_\text{inv}$ \\
\midrule
cubic \\
\ \includegraphics[scale=0.7]{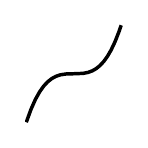} & 3 & $3u + u^3 \bfi $ & $\frac{t-1}{t+1}$ & 1 & 1 & 0.09 & 0.09\\ \midrule
folium\\
\ \includegraphics[scale=0.7]{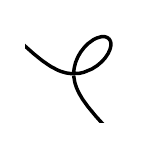} & 3 & $\frac{3u + 3u^2\bfi}{u^3 + 1} $ & $t+1$ & 0 & 1 & 0.23 & 0.29\\ \midrule
epitrochoid \\
\ \includegraphics[scale=0.7]{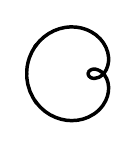} & 4 & $ \frac{1 - 7u^4 + 18u^2 - (20u^3 - 4u)\bfi}{(u^2 + 1)^2} $ & $t - 1$ & 0 & 1 & 0.12 & 0.14 \\ \midrule
3-leaf rose\\
\ \includegraphics[scale=0.7]{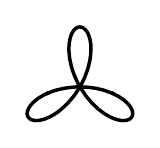} & 4 & $\frac{1 - 3u^2}{(u^2 + 1)^2} (u + \bfi)$ & $t-1$ & 2 & 3 & 0.53 & 0.82\\ \midrule
deltoid\\
\ \includegraphics[scale=0.7]{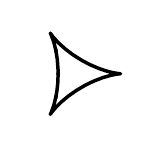} & 4 & $-\frac{u^4 + u^2 - 3 - 8u^3\bfi}{(u^2 + 1)^2}$ & $t+1$ & 2 & 3 & 0.17 & 0.47\\ \midrule
lemniscate\\
\ \includegraphics[scale=0.7]{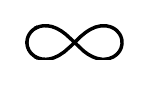} & 4 & $\frac{(1 - u^2)(1 + u^2 + 2u\bfi)}{u^4 + 6u^2 + 1}$ & $t + 2$ & 1 & 3 & 0.15 & 0.25\\ \midrule
astroid\\
\ \includegraphics[scale=0.7]{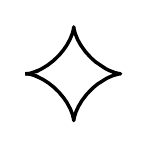} & 6 &
$\frac{(1-u^2)^3 + 8u^3\bfi}{(1+u^2)^3}$ & $\frac{2t-1}{t+2}$ & 3 & 5 & 0.75 & 1.65\\ \midrule
cardioid offset\!\!\!\!\! \\
\ \includegraphics[scale=0.7]{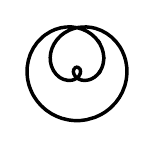} & 8 &
see worksheet \cite{WebsiteGeorg}
& & 0 & 1 & 0.30 & 0.37\\
\bottomrule
\end{tabular*}
\caption{Average CPU times $t_\text{rot}$ and $t_\text{inv}$ (seconds) for determining rotations and involutions of classical plane curves.}\label{tab:planesymmetries}
\end{table}

\begin{table}
\begin{tabular*}{\columnwidth}{m{7em}@{\extracolsep{\stretch{1}}}*{5}{ccccc}@{}}
\toprule
curve & degree & parametrization $\bfx$ & $u(t)$ & \#$_\text{inv}$ & $t_\text{inv}$ \\
\midrule
twisted cubic \\
\includegraphics[scale=0.13]{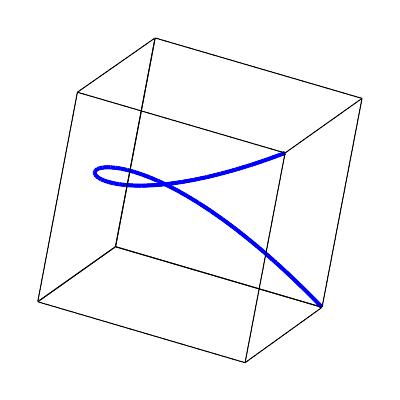} & 3 & $\left( u, u^2, u^3 \right)$ & $ \frac{1}{t+1}$ & 1 & 0.26 \\ \midrule
cusp\\
\includegraphics[scale=0.13]{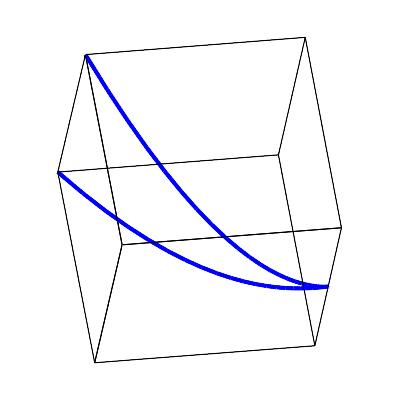}                & 4 & $\left( u^2, u^3, u^4 \right)$ & $\frac{t-1}{t+1}$ & 1 & 0.52\\ \midrule
axial sym. 1\\
\includegraphics[scale=0.13]{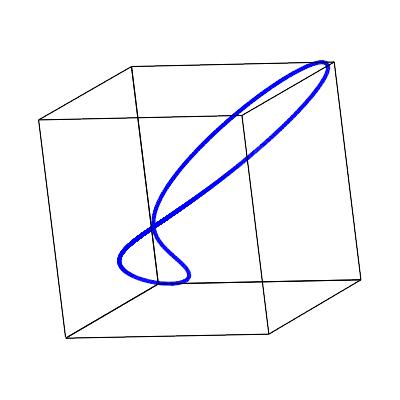}           & 4 & $\left( \frac{u^3+u}{u^4+1},\frac{u^3}{u^4+1},\frac{u^2}{u^4+1}
 \right)$ & $t+1$  &  1 & 2.22 \\ \midrule
crunode\\
\includegraphics[scale=0.13]{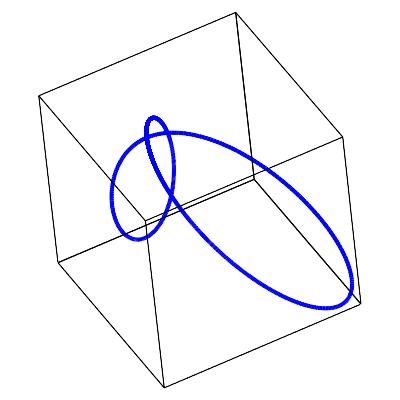}             & 4 & $\left( \frac{u}{1+u^4}, \frac{u^2}{1+u^4}, \frac{u^3}{1+u^4} \right)$ & $t + 2$ & 3 & 39.6\\ \midrule
inversion 1\\
\includegraphics[scale=0.13]{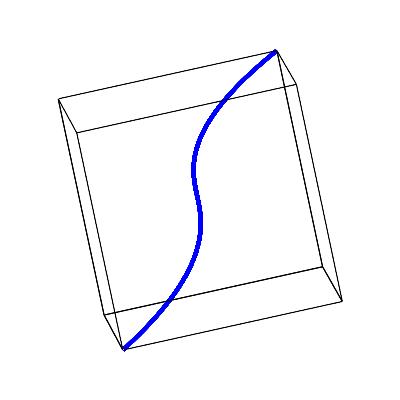} &  7 & $\left( u^7   , u^3 + u, u^5 + u^3 \right)$ & $\frac{t-2}{t+1}$ & 1 & 6.8\\ \midrule
space rose\\
\includegraphics[scale=0.13]{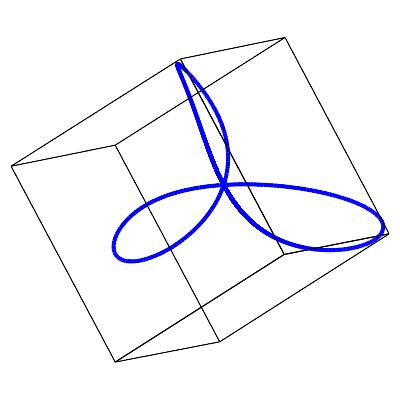} &  8 & $\left(\frac{1-3u^2}{(1+u^2)^2}  , \frac{(1-3u^2)u}{(1+u^2)^2}, \frac{(1-3u^2)u^3}{(1+u^2)^4} \right)$ & $t-1$ & 1 & 57.6 \\ \midrule
inversion 2\\
\includegraphics[scale=0.13]{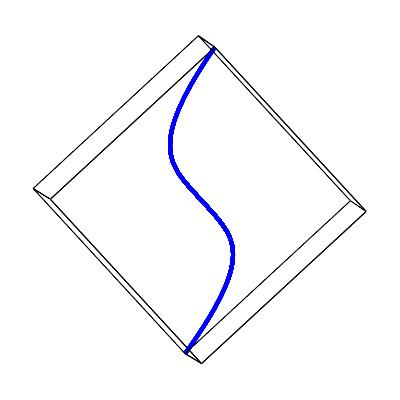} & 11 & $\left( u^{11}, u^3 + u, u^5 + u^3 \right)$ & $\frac{t-2}{t+1}$ & 1 & 75.6\\
\bottomrule
\end{tabular*}
\caption{Average CPU time $t_\text{inv}$ (seconds) for involutions of several space curves.}\label{tab:spaceinvolutions}
\end{table}

\subsection{An example: involutions of the twisted cubic}\label{sec:twistedcubic}
\noindent Let $\CCC$ be the \emph{twisted cubic}, defined parametrically as the image of the map
\begin{equation}\label{eq:TwistedCubic}
\bfx: \RR\longrightarrow \CCC\subset \RR^3,\qquad
t\longmapsto \left( \frac{1}{t+1}, \frac{1}{(t+1)^2}, \frac{1}{(t+1)^3} \right).
\end{equation}
From \eqref{eq:gammaofb} and using that $\Delta = -1 -bc$ we find
\[ c(b) = \frac{(b+2)(25b^4 + 89b^3 + 228b^2 + 262b + 350)}{14(b^4 + 2b^3 + 6b^2 + b + 14)},\]
\[ \Delta(b)=-\frac{(b + 1)(25b + 14)(b^4 + 4b^3 + 10b^2 + 12b + 14)}{14(b^4 + 2b^3 + 6b^2 + b + 14)}.\]
One obtains $Q(b)$ from \eqref{eq:diff1}, \eqref{eq:diff2}, and \eqref{eq:diff1xdiff2} and $\bfb(b)$ from evaluating \eqref{eq:fundam} at $t = 0$. Substituting $Q(b), \bfb(b)$, and $c(b)$ into \eqref{eq:fundam}, one finds that the coefficients of the powers of $t$ in the numerator have gcd
\[ (b+2)(25b + 14)(b^4 + 4b^3 + 10b^2 + 12b + 14)\]
for direct transformations and
\[
(25b + 14)(b^4 + 4b^3 + 10b^2 + 12b + 14)
\]
for opposite transformations. Since the M\"obius transformation is invertible, $\Delta(b)$ is nonzero and there is only one relevant factor $P(b) = b + 2$ for the direct transformations. Substituting $b = -2$ into $\Delta(b), \bfb(b)$, and $c(b)$, one finds the M\"obius transformation $\varphi(t) = -t - 2$ with corresponding axial symmetry
\[
f(\bfx) = Q\bfx + \bfb, \qquad
Q    = \begin{bmatrix} -1 & 0 &  0\\  0 & 1 &  0\\  0 & 0 & -1 \end{bmatrix},\qquad
\bfb = \begin{bmatrix} 0 \\ 0 \\ 0 \end{bmatrix}.
\]
This symmetry is depicted in Figure \ref{fig:TwistedCubic} by connecting corresponding points by lines. There are no additional symmetries for the case $d = 0$.

\subsection{Performance}\label{sec:performance}
\noindent We test the performance of the algorithms in Section \ref{sec:determinesym} for several classical curves on a Dell XPS 15 laptop, with 2.4 GHz i5-2430M processor and 6 GB RAM. Additional technical details are provided in the \Sage~worksheet \cite{WebsiteGeorg}.

For each curve, Tables \ref{tab:planesymmetries} and \ref{tab:spaceinvolutions} list the degree, a standard parametrization, a reparametrization $u(t)$ that brings this curve into general position, the number $\#_\text{rot}$ and $\#_\text{inv}$ of (nontrivial) rotations and involutions found, and the average CPU times $t_\text{rot}$ and $t_\text{inv}$ (seconds) of the computations. In each case the algorithm for finding rotations performs better than the algorithm for finding involutions. The curves ``inversion 1'' and ``inversion 2'' are constructed to have precisely one central inversion.

Note that the algorithm for finding involutions of space curves performs significantly worse for the crunode in Table \ref{tab:spaceinvolutions} than for the other curves of degree four. The reason seems to be that, besides the degree of a parametrization, the sizes of the coefficients greatly influence the performance of the algorithms. Even rational numbers with relatively small numerator and denominator get blown up by simple arithmetic operations. This is a common problem when computing in exact arithmetic.

Most of the computation time is spent by substituting the symmetry and M\"obius transformation in \eqref{eq:fundam} and finding the polynomial conditions on the parameter $b$.

\section{Conclusion}\label{sec:conclusion}
\noindent We have provided effective methods for determining the involution symmetries of a plane or space curve defined by a rational parametrization, and the rotation symmetries of a rational plane or Pythagorean-hodograph space curve. Examples were given, and the algorithms have been implemented in \Sage. Experiments show that we can generally quickly compute the rotation and involution symmetries of curves of relatively low degree. The current approach cannot guarantee to find all the rotation symmetries of a space curve efficiently, except in the case of Pythagorean-hodograph curves. For the general case, an alternative strategy seems to be required.

%\section*{References}

\end{document}